\documentclass[a4paper]{amsart}
\usepackage{amsfonts}
\usepackage{amssymb}
\usepackage{nicefrac}
\usepackage{enumerate}
\usepackage{url}
\usepackage{color}      
\usepackage{capt-of}
\usepackage{graphicx}
\newcommand{\CChi}{\mbox{\Large$\chi$}}
\usepackage{tikz}
\newcommand{\R}{\mathbb{R}}

\newcommand{\N}{\mathbb{N}}
\newcommand{\ep}{\varepsilon}
\newcommand{\T}{\mathbb{T}_m}
\renewcommand{\S}{\mathcal{S}}
\newcommand{\Es}{\mathbb{E}_{S_{I},S_{II}}^{x_0}}

\newcommand{\dis}{\displaystyle}
\newcommand{\F}{\mathcal{F}}

\newtheorem{te}{Theorem}[section]
\newtheorem{lem}[te]{Lemma}

\newtheorem{pr}[te]{Proposition}

\theoremstyle{remark}
\newtheorem{re}[te]{Remark}
\newtheorem{ex}[te]{Example}

\theoremstyle{definition}
\newtheorem{de}[te]{Definition}

\title[Unique continuation for a nonlinear equation on trees]
{The unique continuation property for a nonlinear equation on trees}
\author[L. M. Del Pezzo, C. A. Mosquera and J. D. Rossi]
{Leandro M. Del Pezzo, Carolina A. Mosquera and Julio D. Rossi}

\address{Leandro M. Del Pezzo and Carolina A. Mosquera \hfill\break\indent
CONICET and Departamento  de Matem{\'a}tica, FCEyN, Universidad de Buenos Aires,
\hfill\break\indent Pabellon I, Ciudad Universitaria (1428),
Buenos Aires, Argentina.}

\email{{\tt ldpezzo@dm.uba.ar, mosquera@dm.uba.ar}}

\address{Julio D. Rossi \hfill\break\indent
Departamento  de An{\'a}lisis Matem{\'a}tico, Universidad de Alicante,
\hfill\break\indent Ap. correo 99, 03080, Alicante, SPAIN.
}

\email{{\tt julio.rossi@ua.es}}
\thanks{
Leandro M. Del Pezzo was partially supported by ANPCyT PICT No. 2006-290 and CONICET (Argentina) PIP 5478/1438.
Carolina A. Mosquera was partially supported by UBACyT X638. Julio D. Rossi was partially supported by by projects MTM2010-18128 and MTM2011-27998 (Spain).}

\begin{document}

\begin{abstract}
In this paper we study the game $p-$Laplacian on a tree, that is,
$$
u(x)=\frac{\alpha}2\left\{\max_{y\in \S(x)}u(y) + \min_{y\in
\S(x)}u(y)\right\} + \frac{\beta}{m}\sum_{y\in \S(x)} u(y),
$$
here $x$ is a vertex of the tree and $S(x)$ is the set of
successors of $x$.
 We study the family of the subsets of the tree that enjoy the unique
continuation property, that is, subsets $U$ such that $u\mid_U=0$
implies $u \equiv 0$.
\end{abstract}

\maketitle


\section{Introduction}

Our main goal in this paper is to analyze for which sets the
unique continuation property is valid for the nonlinear equation
known as the game $p-$Laplacian on a tree. This nonlinear equation
reads as follows
\begin{equation}\label{eq.main.intro}
u(x)=\frac{\alpha}2\left\{\max_{y\in \S(x)}u(y) + \min_{y\in
\S(x)}u(y)\right\} + \frac{\beta}{m}\sum_{y\in \S(x)} u(y), \qquad
\forall x\in\T.
\end{equation}
Here $x$ is a vertex of the $m$-branches directed tree $\T$ and $\S(x)$
is the set of successors of that vertex (see Section
\ref{sect-prelim} for details).

Equation \eqref{eq.main.intro} arises naturally when one considers
Tug-of-War games. In fact, let us describe the game that gives
rise to \eqref{eq.main.intro}. This is a zero sum game with two
players in which the earnings of one of them are the losses of the
other. Starting with a token at a vertex $x_0\in\T$, the players
flip a biased coin with probabilities $\alpha$ and $\beta,$
$\alpha+\beta=1$. If the result is a head (probability $\alpha$),
they toss a fair coin to decide who move the token. If the outcome
of the second toss is heads, then Player I moves the token to any
$x_1\in
\S(x_0)$, while in case of tails, Player II moves the token
to any $x_1\in\S(x_0)$. In the other case, that is, if they get
tails in the first coin toss (probability $\beta$), the game state
moves according to the uniform probability density to a
  random vertex $x_1\in \S(x_0)$. They continue playing and
given a continuous function $F:[0,1]\to\R,$ the final payoff is
given by $
\lim_{k\to+\infty} u(x_k)=F(\pi)$. This game has a value $u$ that
  verifies a Dynamic Programming Principle formula, that for this game is
  given by \eqref{eq.main.intro}. This can be intuitively explained
  as follows: the expected value of the game is the sum among all
  possibilities of the expected value in the successors.
  Note that Player I tries to maximize the
  expected value while Player II tries to minimize it. Hence,
  there is $\alpha/2$ probability of each player to win (and hence $\alpha/2$ probability
  to move to the vertex where the maximum is located and $\alpha/2$ to the minimum) and
  $\beta$ probability of the random choice of the next point.
  Formula \eqref{eq.main.intro} encodes all these possibilities.
  See Section \ref{sect-games} for more details concerning the
  game.

Also equation \eqref{eq.main.intro} can be viewed as a combination
(with coefficients $\alpha$ and $\beta$) of the discrete infinity
Laplacian, studied in \cite{s-tree}, that is given by
$$
\frac{1}2\left\{\max_{y\in \S(x)}u(y) + \min_{y\in
\S(x)}u(y)\right\} - u(x)
$$
and the discrete Laplacian that in this case is given by $$
 \frac{1}{m}\sum_{y\in \S(x)} u(y) - u(x).$$

The study of the unique continuation property for solutions of
differential equalities and inequalities of second order elliptic
operators with smooth and non-smooth real coefficients has a large
history and is essentially complete. Let us state a classical
strong unique continuation result for the divergence-form linear
equation
\[
\mbox{div} (A(x) Du) + \langle b(x), Du\rangle + c(x)u =0.
\]

{\it Classical unique continuation property}. Let $\Omega\subset
\R^N$ be a connected domain. Under adequate assumptions on the
coefficients $A$, $b$ and $c$, if $u$ is a weak solution that
vanishes in an open subset of $\Omega$, then $u \equiv 0$ in
$\Omega$.

A general version of this statement is proved by Hormander
\cite{Horm} using Carleman estimates. See also \cite{Gar} that
contains a proof of the result via monotonicity formulas. This
result was recently generalized to fully nonlinear equations
(under some assumptions on regularity of the equation) in
\cite{silv}. For more details and references concerning unique
continuation we refer the reader to
\cite{Counter,JK,KT0,KT,Martio,Wolff}.

Concerning unique continuation for quasilinear problems  like the
$p-$Laplacian, $\mbox{div} (|\nabla u|^{p-2} \nabla u) =0$, in
\cite{Ale}, the author proves the unique continuation
property in the plane for all $1<p<+\infty$, for a different
approach see also \cite{BI,man}.
In the higher dimensions, as far as we know, the problem remains
open for $p\neq2.$ Recently, in \cite{GM1} , the authors deal with
this problem by studying a certain generalization of Almgren's
frequency function for the $p-$Laplacian. Using this approach the
authors have obtained some partial results. See also the reference
\cite{GM2}.

In the case of connected finite graph this problem can be stated
as follows: Let $E$ be a
 connected finite graph.
 We assign to every edge of $E$ length one and we define
$ d(x,y)=\inf_{x\sim y}|x\sim y|$, where $x\sim y$ is the path
connecting vertex $x$ to the vertex $y$ and $|x\sim y|$ is the
number of edges in this path. Assume that $u$ is a solution to
\eqref{eq.main.intro} on $E$ (these functions are also called
$p-$harmonious functions, see
\cite{s-tree}) and that $u=0$ on $B_R(x)$ where $B_R(x)$ is the ball of
radius $R>0$ centered at a node $x$ of $E$ contained within this
graph. Does it imply that $u\equiv 0$ on $E$? The answer to this
problem is negative, see examples in Section 3.6 of \cite{s-tree}.
Also, in \cite{s-tree}, the author proves the existence and
uniqueness and a comparison principle for the Dirichlet problem
for
\eqref{eq.main.intro} in the case of a connected finite graph and in the case where the graph is
$\mathbb{T}_3.$

\subsection{Main results} Our results can be summarized as follows: first, for a general
$m$-branches directed tree, we prove existence, uniqueness and a
comparison principle for the Dirichlet problem for
\eqref{eq.main.intro}. In addition we present an approximation scheme
that can be used to approximate numerically the solution when the
boundary data is a Lipschitz function. Next, we prove our main
result, that is a description of sets $U \subset \T$ for
which the unique continuation property holds. As we have
mentioned, this means that any bounded solution to
\eqref{eq.main.intro} that vanishes on $U$ vanishes everywhere in
$\T$.

{\bf Organization of the paper.} In Section
\ref{sect-prelim} we collect some preliminary facts concerning
trees and solutions to
\eqref{eq.main.intro}; in Section~\ref{sect-games} we describe with some details the
associated Tug-of-War game  and use it to prove existence and
uniqueness for the Dirichlet problem and a comparison principle
for solutions to
\eqref{eq.main.intro}; in Section~\ref{sect-aproximacion} we
present a numerical scheme that approximates solutions to
\eqref{eq.main.intro} and, finally, in Section~\ref{sect-unique.cont} we prove our
main result concerning the sets for which unique continuation
holds.


\section{Preliminaries} \label{sect-prelim}

\subsection{Directed Tree}
Let $m\in\mathbb{N}_{>2}$. In this work we consider a directed
tree $\T$ with regular $m-$branching, that is, $\T$ consists of
the empty set $\emptyset$ and all finite  sequences
$(a_1,a_2,\dots,a_k)$ with $k\in\N,$ whose coordinates $a_i$ are
chosen from $\{0,1,\dots,m-1\}.$ The elements in $\T$ are called
vertices. Each vertex $x$ has $m$ successors, obtained by adding
another coordinate. As we mentioned in the introduction, we will
denote by $\S(x)$ the set of successors of the vertex $x.$ A
vertex $x\in\T$ is called a $n-$level vertex ($n\in\mathbb{N}$) if
$x=(a_1,a_2,\dots,a_n).$  The set of all $n-$level vertices is
denoted by $\T^n.$

\begin{ex}
  Let $\kappa\in\mathbb{N}_{\ge3}.$
  The $\nicefrac{1}{\kappa}-$Cantor set, that we denote by
  $C_{\nicefrac{1}{\kappa}}$,
is the set of all $x\in[0,1]$ that have a base $\kappa$ expansion
without the digit $1$,
  that is $x=\sum a_j\kappa^{-j}$ with
  $a_j\in\{0,1,\dots,\kappa-1\}$ with $a_j \neq 1$. Thus $C_{\nicefrac{1}{\kappa}}$
  is obtained from $[0,1]$ by removing the second $\kappa-$th part of the line segment $[0,
  1]$, and then removing the second interval of length $\nicefrac1{\kappa}$
  from the remaining intervals, and so on. This set
  can be thought of as a directed tree with regular $m-$branching with $m=\kappa-1$.

  For example, if $\kappa=3$, we identify $[0, 1]$ with $\emptyset,$
  the sequence $(\emptyset, 0)$ with the first interval right $[0,
  \nicefrac13]$,
  the sequence $(\emptyset, 1)$ with the central interval $[\nicefrac13,
  \nicefrac23]$ (that is removed),
  the sequence $(\emptyset, 2)$ with the left interval  $[\nicefrac23, 1],$
  the sequence $(\emptyset, 0 ,0)$ with the interval
  $[0, \nicefrac{1}{9}]$ and so on.
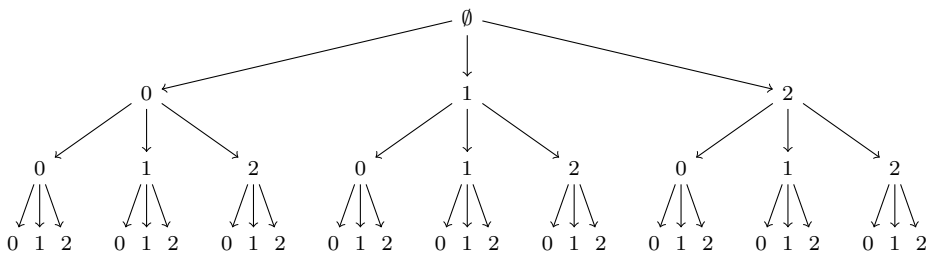
\begin{center}
\begin{tikzpicture} [font=\footnotesize,
grow=down, 
level 1/.style={->, sibling distance=12em},
level 2/.style={->,sibling distance=4em}, level distance=1cm,
level 3/.style={->,sibling distance=1em}, level distance=1cm]

    \node {$\emptyset$}
        child { node {0}
						child { node{0}
										child { node{0}}
										child { node{1}}
                        					child { node{2}}
								}
						child { node{1}
										child { node{0}}
										child { node{1}}
                        					child { node{2}}
								}
                        					child { node{2}
										child { node{0}}
										child { node{1}}
                        					child { node{2}}
								}
				}
       child { node {1}
						child { node{0}
										child { node{0}}
										child { node{1}}
                        					child { node{2}}
								}
						child { node{1}
										child { node{0}}
										child { node{1}}
                        					child { node{2}}
								}
                        					child { node{2}
										child { node{0}}
										child { node{1}}
                        					child { node{2}}
								}
				}
 child { node {2}
						child { node{0}
										child { node{0}}
										child { node{1}}
                        					child { node{2}}
								}
						child { node{1}
										child { node{0}}
										child { node{1}}
                        					child { node{2}}
								}
                        					child { node{2}
										child { node{0}}
										child { node{1}}
                        					child { node{2}}
								}
				}
    ;
\end{tikzpicture}
\end{center}
\end{ex}

\medskip

A branch of $\T$ is an infinite sequence of vertices, each followed by its immediate successor.
The collection of all branches forms the boundary  $\partial\T$ of $\T.$

\medskip

We now define a metric on $\T\cup \partial\T.$ The distance
between two sequences (finite or infinite) $\pi=(a_1,\dots,
a_k,\dots)$ and $\pi'=(a_1',\dots, a_k',\dots)$ is $m^{-K+1}$ when
$K$ is the first index $k$ such that $a_k\neq a_k';$ but when
$\pi=(a_1,\dots, a_K)$ and $\pi'=(a_1,\dots, a_K,
a_{K+1}',\dots),$ the distance is $m^{-K}.$  Hausdorff measure and
Hausdorff dimension are defined using this metric. We can observe
that $\T$ and $\partial\T$ have diameter one and $\partial\T$ has
Hausdorff dimension one. Now, we observe that the mapping
$\psi:\partial\T\to[0,1]$ defined as
\[
\psi(\pi):=\sum_{k=1}^{+\infty} \frac{a_k}{m^{k}}
\]
is surjective, where $\pi=(a_1,\dots, a_k,\dots)\in\partial\T$ and
$a_k\in\{0,1,\dots,m-1\}$ for all $k\in\mathbb{N}.$ Whenever
$x=(a_1,a_2,\dots,a_k)$ is a vertex, we set
\[
 \psi(x):=\psi(a_1,a_2,\dots,a_k,0,\dots,0,\dots).
\]
We can also associate to a vertex $x=(a_1,a_2,\dots,a_k)$ an
interval $I_x$ of length $\frac{1}{m^k}$ as follows
\[
 I_x=\left[\psi(x),\psi(x)+\frac1{m^k}\right].
\]
Observe that for all $x\in \T$, $I_x \cap \partial\T$ is the
subset of $\partial\T$ consisting of all branches that start at
$x$.

With an abuse of notation, we will write
$\pi=(x_1,\dots,x_k,\dots)$
instead of $\pi=(a_1,\dots,a_k,\dots)$ where $x_1=a_1$ and
$x_k=(a_1,\dots,a_k)\in\S(x_{k-1})$ for all $k\in\N_{\ge2}.$

\subsection{$p-$harmonious functions} Inspired in \cite{s-tree} and \cite{MPR3}
we give the definition of the $p-$harmonious function that we will
consider throughout this paper.

\begin{de} Let $\alpha,\beta>0$ such that $\alpha+\beta=1$.
A function $u:\T\to\R$ is called $p-$subharmonious if
\begin{equation*}
u(x)\le\frac{\alpha}2\left\{\max_{y\in \S(x)}u(y) + \min_{y\in \S(x)}u(y)\right\}
+ \frac{\beta}{m}\sum_{y\in \S(x)} u(y) \quad \forall x\in\T,
\end{equation*}
and $p-$superharmonious if the opposite inequality holds for all
$x\in\T$. We say that $u$ is $p-$harmonious if $u$ is both
$p-$subharmonious  and $p-$superharmonious.
\end{de}

\begin{re}
If $u$ is a $p-$harmonious function on $\T,$ then
$u^+=\max\{u,0\}$ and $u^-=\max\{-u,0\}$ are $p-$subharmonious
functions on $\T.$
\end{re}

Next, we collect some properties of $p-$harmonious functions.

\begin{lem}If $u$ is a $p-$subharmonious function bounded above on $\T$ and
there exists $x\in\T$ such that $u(x)=\max_{y\in\T} u(y)$ then
$u(y)=u(x)$ for any $y\in\T$ such that $I_y\subset I_x.$
\end{lem}

\begin{proof} Throughout this proof let $M=u(x)=\max_{y\in\T} u(y)$. We first
observe that it is sufficient to show that $u(y)=M$ for all
$y\in\S(x).$ Since $u$ is $p-$subharmonious on $\T,$ we have that
\begin{align*}
M=u(x)&\le\frac{\alpha}2\left\{\max_{y\in \S(x)}u(y) + \min_{y\in \S(x)}u(y)\right\}
+ \frac{\beta}{m}\sum_{y\in \S(x)} u(y)\\
&\le\left(\frac{\alpha}2+\frac{(m-1)\beta}m \right)M +\left(\frac{\alpha}2
+\frac{\beta}{m}\right) \min_{y\in \S(x)}u(y).
\end{align*}
Then
\[
\left(\frac{\alpha}2+\frac{\beta}{m}\right) M\le\left(\frac{\alpha}2
+\frac{\beta}{m}\right) \min_{y\in \S(x)}u(y).
\]
Therefore $u(y)=u(x)$ for all $y\in\S(x).$
\end{proof}

In the same manner, we can prove the following lemma

\begin{lem}If $u$ is a $p-$superharmonious function bounded below
on $\T$ and there exists $x\in\T$ such that $u(x)=\min_{y\in\T}
u(y),$ then $u(y)=u(x)$ for any $y\in\T$ such that $I_y\subset
I_x.$
\end{lem}

Now we show that $p-$harmonious functions are well behaved with
respect to uniform convergence.

\begin{lem}\label{unifconv} The uniform limit of a sequence of
$p-$harmonious functions is a $p-$harmonious function.
\end{lem}
\begin{proof}
Let $\{u_n\}_{n\in\N}$ be a sequence of $p-$harmonious functions
which converges uniformly to $u.$ We will show that $u$ is a
$p-$harmonious function. Given $\ep>0,$ there exists
$n_0=n_0(\ep)$ such that if $n\ge n_0,$
\begin{equation}\label{limite}
 |u(x)-u_{n}(x)|\le \ep \quad \forall x\in\T.
\end{equation}
Then, for all $x\in\T$ and $n\ge n_0$ we have that
\[
u_n(y)-\ep\le u(y)\le u_n(y)+\ep\quad \forall y\in\S(x).
\]
Thus, for all $x\in\T$ and $n\ge n_0,$
\begin{align*}
u_n(x)-\ep&=\frac{\alpha}2\left\{\max_{y\in \S(x)}u_n(y)
+ \min_{y\in \S(x)}u_n(y)\right\} + \frac{\beta}{m}\sum_{y\in \S(x)} u_n(y)-\ep\\
&\le \frac{\alpha}2\left\{\max_{y\in \S(x)}u(y) + \min_{y\in \S(x)}u(y)\right\}
+ \frac{\beta}{m}\sum_{y\in \S(x)} u(y)\\
&\le \frac{\alpha}2\left\{\max_{y\in \S(x)}u_n(y) + \min_{y\in \S(x)}u_n(y)\right\}
+ \frac{\beta}{m}\sum_{y\in \S(x)} u_n(y)+\ep\\
&=u_n(x)+\ep.
\end{align*}
Taking limit as $n\to+\infty,$ we get that
\[
u(x)-\ep
\le \frac{\alpha}2\left\{\max_{y\in \S(x)}u(y) + \min_{y\in \S(x)}u(y)\right\}
+ \frac{\beta}{m}\sum_{y\in \S(x)} u(y)
\le u(x)+\ep \quad\forall x\in\T.
\]
Then, since $\ep$ is arbitrary, we have that
\[
u(x)
=\frac{\alpha}2\left\{\max_{y\in \S(x)}u(y) + \min_{y\in \S(x)}u(y)\right\}
+ \frac{\beta}{m}\sum_{y\in \S(x)} u(y)
\quad\forall x\in\T,
\]
that is, $u$ is a $p-$harmonious function.
\end{proof}

The Fatou set $\F(u)$ of a function $u$ is the set of the branches
$\pi=(x_1,\dots, x_k, \dots)$ on which
\[
\lim_{k\to+\infty}u(x_k)
\]
 exists and is finite, and $BV(u)$ is the set of the branches
 $\pi=(x_1,\dots, x_k, \dots)$ on which $u$ has finite variation
 \[
 \sum_{k=1}^{\infty}|u(x_{k+1})-u(x_k)|.
 \]
Clearly $BV(u)\subseteq\F(u).$

Now we use the results of \cite{KLW} to show that the infimum of
Hausdorff dimension of $BV(u)$ and $\F(u)$ are equal over all
bounded $p-$harmonious functions on $\T.$

\begin{te}
Let $\mathcal{H}^m$ be the set of bounded $p-$harmonious functions
on $\T.$ Then
\begin{equation}\label{Fatou}
\min_{\mathcal{H}^m}{\rm dim}\, \F(u)=\min_{\mathcal{H}^m}{\rm dim}\, BV(u)
=\frac{
\log\left(\gamma^{-\frac{m\alpha+2(m-1)\beta}{2m}}+(m-1)\gamma^{\frac{m\alpha+2\beta}{2m}}\right)}{\log m},
\end{equation}
where $$\gamma=\frac{m\alpha+2(m-1)\beta}{(m-1)(m\alpha+2\beta)}$$
and ${\rm dim}$ denotes the usual Hausdorff dimension.
\end{te}

\begin{proof}
By Theorem A in \cite{KLW} we have that
\[
\min_{\mathcal{H}^m}\mbox{ dim } \F(u)=\min_{\mathcal{H}^m}\mbox{ dim } BV(u)
=\frac{\log f(m)}{\log m },
\]
where
\[
f(m)=\min\left\{\sum_{j=1}^m e^{x_j}\colon x\in\mathbb{R}^m \mbox{ s. t.} \frac{\alpha}{2}
\left(\max_{1\le j\le m} x_j +\min_{1\le j\le m}x_j\right) +\frac{\beta}{m}\sum_{j=1}^m x_j=0
\right\}.
\]

We obseve that the minimum $f(m)$ is attained at
\[
x_1=-\frac{\alpha m + 2(m-1)\beta}{2m}\log\gamma,\qquad x_j=
\frac{m\alpha+2\beta}{2m}\log\gamma, \qquad 2\le j\le m,
\]
with value
\[
\gamma^{-\frac{m\alpha+2(m-1)\beta}{2m}}+(m-1)\gamma^{\frac{m\alpha+2\beta}{2m}},
\]
which completes the proof.
\end{proof}

\begin{re}
In \cite{KW}, for the classical discretization of the
$p-$harmonic function on trees,
$$
\sum_{y \in S(x)} |u(x)-u(y)|^{p-2} (u(x)-u(y)) =0,
$$
 the authors prove that
$$\lim_{m\to+\infty}
\min_{\mathcal{H}^m}{\rm dim}\, \F(u)=\lim_{m\to+\infty}\min_{\mathcal{H}^m}{\rm
dim}\, BV(u)=1 $$ for all $p>1$. In our case, we can observe that,
when $\alpha=0,$ we have that $\gamma=1$ and therefore
\[
\min_{\mathcal{H}^m}{\rm dim}\, \F(u)=\min_{\mathcal{H}^m}{\rm dim}\, BV(u)=1
\]
for all $m\in\N_{\ge2}.$ On the other hand, when $\alpha\neq0,$ if
we rewrite \eqref{Fatou} as
$$
\left(\frac{1}{2}+\frac{(m-2)\beta}{2m}\right)\left(1
+\frac{\log\left(\frac{(1-\frac{1}{m})(\alpha+\frac{2\beta}{m})}
{(\alpha+\frac{2(m-1)\beta}{m})}\right)}{\log(m)}\right)+
\frac{\log\left(\frac{2}{\left(\alpha+\frac{2\beta}{m}\right)}\right)}{\log(m)}
$$
and take limit as $m\to+\infty,$ we obtain that
\[
\lim_{m\to+\infty}\min_{\mathcal{H}^m}{\rm dim}\, \F(u)=\lim_{m\to+\infty}
\min_{\mathcal{H}^m}{\rm dim}\, BV(u)= \frac{1}{2}+\frac{\beta}{2}.
\]
\end{re}

\section{The Dirichlet Problem and a Tug-of-War Game} \label{sect-games}

First, let us introduce what we understand by the Dirichlet
problem for $p-$harmonious functions.

\noindent{\bf Dirichlet Problem}
$(DP)$. Given $\alpha,\beta>0$ such that $\alpha+\beta=1$ and a
continuous function $F:[0,1]\to\R,$ find a $p-$harmonious function
$u$ such that
\[
\lim_{k\to+\infty} u(x_k)=F(\pi) \quad \forall \pi=(x_1,\dots,x_k,\dots)\in\partial\T.
\]
We say that $v$ is a supersolution of $(DP)$ if $v$ is
$p-$superharmonious and
\[
\lim_{k\to+\infty} v(x_k)\ge F(\pi) \quad \forall \pi=(x_1,\dots,x_k,\dots)\in\partial\T.
\]
We say that $v$ is a subsolution of $(DP)$ if $v$ is
$p-$subharmonious and
\[
\lim_{k\to+\infty} v(x_k)\le F(\pi) \quad \forall \pi=(x_1,\dots,x_k,\dots)\in\partial\T.
\]

First, we want to show that the $(DP)$ has a unique solution. To
this end we use the Tug-of-War game introduced in \cite{PSSW}, see
also
\cite{MPR3}. Now we describe the game and refer to \cite{ms1} for
more details and references. It is a two player zero sum game.
Starting with a token at a vertex $x_0\in\T$, the players flip a
biased coin with probabilities $\alpha$ of getting a head and
$\beta$ of a tail, $\alpha+\beta=1$. If they get a head
(probability $\alpha$), they toss a second coin (a fair coin this
time with probabilities $1/2$ and $1/2$) to decide who move the
token. If the outcome of the second toss is heads, then Player I
moves the token to any $x_1\in
\S(x_0)$. In the case of tails, Player II gets to move the token
to any $x_1\in\S(x_0)$. In the other case, that is, if they get
tails in the first coin toss (probability $\beta$), the game state
moves according to the uniform probability density to a
  random vertex $x_1\in \S(x_0)$.
They continue playing the game forever, generating an infinite
sequence $\pi=(x_0,x_1,\dots,x_k,\dots)$ where $x_k\in\S(x_{k-1})$
for any $k\in\mathbb{N},$ therefore $\pi\in\partial\T.$ Then
Player~I receive from  Player~II the amount $F(\pi)$, where $F$ is
a continuous function from $[0,1]$ to $\R$. This is the reason why
we will refer to $F$ as the final payoff function. Now we define
the expected payoff for an individual game. First, a strategy
$S_I$ for Player I is a collection of measurable mappings $S_I=
\{S_I^{k}\}_{k\in\N}$ such that the next game position is given by
\[
 S_I^{k+1}(x_0, x_1, \dots, x_k)= x_{k+1}\in \S(x_k)
\]
if Player I wins the toss given a partial history $(x_0, x_1,
\dots, x_k).$ Similarly, Player II plays according to a strategy $S_{II}.$
We can observe that the next game position $x_{k+1}\in S(x_k),$
given a partial history $(x_0, \ldots, x_k),$ is distributed
according to the probability
\[
q_{S_I, S_{II}}(x_0, \ldots, x_k, A)= \frac{\alpha}{2}
\delta_{S_I^k(x_0, \ldots, x_k)}(A) + \frac{\alpha}{2}\delta_{S_{II}^k(x_0, \ldots, x_k)}(A)
+  \frac{\beta }{m}\#(A\cap S(x_k)),
\]
where $A$ is a subset of $\T$ and $\#(A\cap S(x_k))$ denotes the
cardinal of the set $A\cap S(x_k).$ Strategies $S_I$ and $S_{II}$
together with an initial state $x_0$ determine a unique
\mbox{probability} measure $\mathbb{P}_{S_I, S_{II}}^{x_0}$ in
$[0, 1].$ For the precise definition of  $\mathbb{P}_{S_I,
S_{II}}^{x_0}$ we refer to
\cite{MPR2}. We define the expected payoff of an individual game
as
\[
 \Es[F]= \int_0^1 F(y)\, \mathbb{P}_{S_I, S_{II}}^{x_0}(dy).
\]
We also define the value of the game for Player I as
\[
 u_I(x_0)= \sup_{S_I}\inf_{S_{II}}\Es[F]
\]
and the value of the game for Player II as
\[
 u_{II}(x_0)= \inf_{S_{II}}\sup_{S_{I}}\Es[F].
\]
The value $u_I(x_0)$ and $u_{II}(x_0)$ are in a sense the best
expected outcomes each player can almost guarantee when the game
starts at $x_0.$ For more details on values of games, we refer to
\cite{msb, s-tree}.

The following theorem states that the game has a value, i.e.
$u_I=u_{II},$ and this value is a solution of $(DP)$.  For a
detailed proof of the existence of a value see \cite{ms1} and, by
an argument completely similar to the proof of Theorem 3.4 in
\cite{MPR2}, we have that the game value is a solution of $(DP)$.

\begin{te}\label{valor}
Let $F:[0,1]\to\R$ be a continuous function. Then the game with
payoff function $F$ has a value $u.$ Furthermore, $u$ is a
solution of $(DP)$ with boundary data $F$.
\end{te}

To see the form of game values $u$ (solution of $(DP)$) let us
mention that in
\cite{s-tree}, an explicit formulae for $\mathbb{P}_{S_I,
S_{II}}^{x_0}$ is given when $F$ is monotone, and therefore we
have an explicit formulae for $u$. In the next section, we will
show how to approximate $u$ in the general case.

From now on, we assume that $F:[0,1]\to\R$ is  a continuous
function. Next we show a comparison principle.

\begin{te}\label{com1} Let $G:[0,1]\to\R$ be a continuous function and $v$
be a bounded supersolution of $(DP)$ with boundary data $G$ such
that $G\ge F$ in $[0,1],$ then
\[
v(x)\ge u(x)
\]
for any $x\in\T,$ where $u$ is the value of game with final payoff
function $F.$
\end{te}

\begin{proof}
First, we show that by choosing a strategy according to the
minimal values of $v,$ Player II can make the process a
supermartingale. More precisely,  Player I follows any strategy
and Player II follows the following strategy, that we will call
$S_{II}^{0}:$ at $x_{k-1}\in\T$ he chooses to step to a vertex
that minimizes $v,$ i.e. a vertex $x_k\in\S(x_{k-1})$ such that
\[
 v(x_k)= \min_{y\in\S(x_{k-1})} v(y).
\]
We start from a vertex $x_0.$ Using that $v$ is a supersolution of
$(DP)$ and the estimated the strategy of Player I by the supremum,
we have that
$$
\begin{array}{l}
\displaystyle \mathbb{E}_{S_I,S_{II}^0}^{x_0}[v(X_k)|x_0,\dots, x_{k-1}]
\\[10pt]
\displaystyle \le
\frac{\alpha}{2}\left\{\min_{y\in\S(x_{k-1})}v(y)+ \max_{y\in\S(x_{k-1})}v(y)\right\}
+ \frac{\beta}{m}\sum_{y\in\S(x_{k-1})}v(y)\\[10pt]
\displaystyle \le v(x_{k-1}),
\end{array}
$$
where $X_k$ is the coordinate process
defined by
\[
X_k(\omega):= x_k \mbox{ for }
\omega=(x_0,\dots,x_k,\dots) \in \T\times\T\times\cdots.
\]

Thus $M_k= v(X_{k})$ is a supermartingale. From this fact, using
Theorem 4.2.2 in \cite{ms1}, the Optional Stopping Theorem,
and that $G\ge F$ in $[0,1],$ we get the desired result.
\end{proof}

Moreover, we have an analogous result for bounded subsolutions of $(DP)$.

\begin{te}\label{com2} Let $G:[0,1]\to\R$ be a bounded function and $v$
be a bounded subsolution of $(DP)$ with boundary data $G$ such
that $G\le F$ in $[0,1],$ then
\[
v(x)\le u(x)
\]
for any $x\in\T,$ where $u$ is the value of the game with final
payoff function $F.$
\end{te}

\begin{proof}
The proof is similar to the previous one.
\end{proof}

Then, we arrive to the main result of this section.

\begin{te}\label{eu} There exists a unique bounded solution of  $(DP)$
with given \mbox{boundary} data $F.$ Moreover, it coincides with the value of the game.
\end{te}

\begin{proof}
Theorem \ref{valor} gives that the value of the game is a solution
of $(DP)$. This proves existence. Theorems
\ref{com1} and \ref{com2} imply uniqueness.
\end{proof}

The above theorem, together with Theorems \ref{com1} and \ref{com2},
give the Comparison Principle for solutions of $(DP)$.

\begin{te}[Comparison Principle]\label{CP}
Let $F, G:[0,1]\to\R$ be  bounded functions.
If $v$ is a bounded  supersolution (subsolution) of $(DP)$ with boundary data $G,$ $u$ is the solution of $(DP)$
with boundary data $F$ and $F\le G$ ($F\ge G$) in $[0,1]$, we have that $u\le v$ ($u\ge v$) in $\T.$
\end{te}


\section{A Numerical Approximation}\label{sect-aproximacion}

In this section we give a numerical approximation for the
solutions of $(DP)$ when the boundary datum $F$  is a continuous
function.

Let $F$ be a real-valued function on $[0,1]$ and $n\in\N,$ we
define $F_n:[0,1]\to\R$ as
\[
F_n(t)=\sum_{j=0}^{m^n-1}
F(t_{nj})\CChi_{I_{nj}}(t)
\]
where $t_{nj}=\frac{j}{m^n},$
$I_{nj}=[t_{nj},t_{n(j+1)})$
for all $j\in\{0,\dots,m^n-2\}$  and
$I_{n(m^n-1)}=[t_{n(m^n-1)},1]$. Note that
this function is piecewise constant.

Our next goal is to construct a $F$-harmonic function $u_n$
such that $u_n(x)=F_n(x)$ for all $x\in\T^k$ for any $k\ge n.$

We first observe that, for all
$j\in\{0,\dots,m^n-1\}$  there exists  $x_{nj}\in\T^n$ such that
$I_{x_{nj}}=\overline{I_{nj}}.$ Then, for all
$k\in\{1,\dots,n\}$, we take
$\{x_{(n-k)j}\}_{j=0}^{m^{n-k}-1}\subset\T$ such that
\[
\S(x_{(n-k)j})=\{x_{(n-k+1)\tau}\colon 1+(j-1)m\le\tau\le jm\} \quad\forall j\in\{0,\dots,m^{n-k}-1\}.
\]

Let $u_n:\T\to\R$ such that
\[
 u_n(y)= F(t_{nj}) \quad\forall y\in\T \mbox{ such that }
 I_y\subset I_{x_{nj}} \mbox{ for some } {j\in\{1,\dots,m^n-1
 \}},
\]
and for any $k\in\{1,\dots,n\}$
\[
u_n(x_{(n-k)j})=\frac{\alpha}2\left\{\max_{y\in \S(x_{(n-k)j})}u(y)
+ \min_{y\in \S(x_{(n-k)j})}u(y)\right\} + \frac{\beta}{m}\sum_{y\in \S(x_{(n-k)j})} u(y)
\]
for all $j\in\{0,\dots,m^{n-k}-1\}$. It is easy to check that
$u_n$ is a $p-$harmoniuous function. Moreover,
if $F$ is bounded then $\{u_n\}_{n\in\N}$ is uniformly bounded on
$\T.$

\begin{re}\label{cu}
Let $F$ be a continuous function on $[0,1].$ Then, given $\varepsilon>0$
there exists $\delta=\delta(\varepsilon)>0$ such that
\[
|F(x)-F(y)|\le \frac{\varepsilon}{2}+\frac{2\|F\|_{\infty}}{\delta}|x-y|
\]
for all $x,y\in[0,1].$
\end{re}

We are now ready to state the main result of this section.

\begin{te}
Let $F:[0,1]\to\R$ be a continuous function. Then the sequence
$\{u_n\}_{n\in\N}$ converges uniformly to the solution $u$ of
$(DP)$ with boundary data $F.$ Moreover, if $F$ is a Lipschitz
function we have a bound for the error, it holds that
\[
|u_{n}(x)-u(x)|\le \frac{ L}{m^n}
\]
for all $x\in\T,$ where $L$ is the Lipschitz constant of $F.$

\end{te}

\begin{proof} We present two proofs of this result. The first
proof only uses game theory to show uniqueness and can be viewed
as an alternative way to prove existence of a solution.

This first proof we will  be divided into 4 steps.

\noindent {\bf Step 1.} Since $F$ is a continuous function on $[0,1],$
by Remark \ref{cu}, given $\varepsilon>0$ there exists
$\delta=\delta(\varepsilon)>0$ such that
\[
|F(x)-F(y)|\le \frac{\varepsilon}{2}+\frac{2\|F\|_{\infty}}{\delta}|x-y|
\]
for all $x,y\in[0,1].$ Therefore, for all $n\in\N$ we have that
\begin{equation*}
|F_n(x)-F(y)|\le \frac{\varepsilon}{2}+\frac{2\|F\|_{\infty}}{\delta m^n}
\quad \forall x,y\in I_{nj} \quad\forall j\in\{0,\dots,m^n-1
\}.
\end{equation*}
Then $\{F_n\}_{n\in\N}$ converges uniformly to $F.$

\noindent {\bf Step 2.} We will prove that $\{u_n\}_{n\in\N}$ is an uniformly Cauchy sequence.

Let $h, k,n\in\N$ and $x\in\T^h.$ If $n\le k\le h,$  there exist
$i\in\{0,\dots,m^n-1\}$ and $j \in\{0,\dots,m^k-1\}$ such that
$u_n(x)=F(t_{ni})$ and $u_k(x)=F(t_{kj}).$ Moreover $I_{x}\subset
I_{x_{kj}}\subset I_{x_{ni}}.$ Then, given $\varepsilon>0,$  using
Remark \ref{cu}, we have that
\begin{equation*}
|u_{n}(x)-u_{k}(x)|\le |F(t_{ni})-F(t_{kj})| \le \frac{\varepsilon}{2}
+\frac{2\|F\|_{\infty}}{\delta m^n} \quad \forall x\in \T^h.
\end{equation*}
Thus, there exists $n_0$ such that if $n\ge n_0,$
\begin{equation}\label{laprima}
|u_{n}(x)-u_{k}(x)|\le \varepsilon \quad \forall x\in \T^h.
\end{equation}
For all $x\in\T^{k-1},$ by \eqref{laprima}, we have that
\[
u_{k}(y)-\ep\le u_n(y) \le u_{k}(y)+\ep \quad\forall y\in\S(x).
\]
Then
\[
u_{k}(x)-\ep\le u_n(x) \le u_{k}(x)+\ep \quad\forall  x\in\T^{k-1},
\]
i.e.,
\[
|u_{n}(x)-u_{k}(x)|\le \ep  \quad\forall  x\in\T^{k-1}.
\]
In the same manner, in $k-1-$steps, we can see that
\[
|u_{n}(x)-u_{k}(x)|\le \ep  \quad\forall  x\in\T.
\]
Therefore $\{u_n\}_{n\in\N}$ is an uniformly Cauchy sequence.

\noindent {\bf Step 3.}  Now, we will show that
\[
u(x)=\lim_{n\to+\infty} u_n(x) \quad\forall x\in\T
\]
is the solution of $(DP)$ with boundary data $F.$

By step 2, $\{u_n\}_{n\in\N}$ converges uniformly to $u.$ Therefore,
 by Lemma \ref{unifconv}, $u$ is a $p-$harmonious function. Then we only need to show that
\[
\lim_{k\to+\infty} u(x_k)=F(\pi) \quad \forall \pi=(x_1,\dots,x_k,\dots)\in\partial\T.
\]

Let $\ep>0$ and  $\pi=(x_1,\dots,x_k,\dots)\in\partial\T.$
Since $\{u_n\}_{n\in\N}$ converges uniformly to $u,$ there exists
$n_0=n_0(\ep)$ such that
\begin{equation}
|u_{n}(x_j)-u(x_j)|<\frac{ \ep}{2} \quad\forall j\in\N,
\label{desc1}
\end{equation}
if $n\ge n_0.$ On the other hand, we can observe that there
exists $n_1=n_1(\ep)$ such that
\begin{equation}
|F_{n}(\pi)-F(\pi)|<\frac{ \ep}{2}
\label{desc2}
\end{equation}
if $n\ge n_1.$
Then, since $u_n(x)=F_n(x)$ for all $x\in\T^j$
for any $j\ge n,$ if $n\ge n_1$ we have that
\begin{equation}
		\label{desconunif2}|u_n(x_j)-F(\pi)|
		\le\frac{\varepsilon}2 \quad\forall j\ge n.
\end{equation}
Finally, taking $n\ge\max\{n_0,n_1\}$ and $j\ge n,$ by
\eqref{desc1} and \eqref{desconunif2}, we get
\[
|u(x_j)-F(\pi)|\le|u(x_j)-u_n(x_j)|+|u_n(x_j)-F(\pi)|\le
\varepsilon.
\]
Therefore,
\[
\lim_{k\to+\infty} u(x_k)=F(\pi) \quad \forall \pi=(x_1,\dots,x_k,\dots)\in\partial\T.
\]

\noindent {\bf Step 4.}
We observe that if $F$ is a Lipschitz function, in the same manner
as in step~2, we obtain that, if $k,n\in\N,$
\[
|u_n(x)-u_k(x)|\le \frac{L}{m^n} \quad\forall x\in\T.
\]
Therefore,
\[
|u_n(x)-u(x)|\le \frac{L}{m^n} \quad\forall x\in\T,
\]
where $L$ is the Lipschitz constant of $F.$ This completes the
first proof.

\medskip

Now we proceed with the second proof of this result. This proof is
shorter but we use here the existence and comparison results
proved in the previous section using game theory.

Using that $F_n$ is a continuous
function on $I_{x_{nj}}$ for all $j\in\{0, \dots,m^n-1\}$
(step 1),
$\{F_n\}_{n\in\N}$ converges uniformly to $F$ and Theorem \ref{CP}, we have that given $\varepsilon>0,$ there exists
$n_0=n_0(\varepsilon)\in\N$ such that for any $n\ge n_0$
\[
 u_n(x)-\varepsilon\le u(x)\le u_n(x)+ \varepsilon
\]
for all $x\in\T$ such that $I_x\subset I_{x_{nj}}$ for some
$j\in\{0,\dots,m^n-1\},$ where $u$ is the solution of $(DP)$ with
boundary data $F$. By the above inequality and using that $u_n$
and $u$ are $p-$harmonious functions, we have that
\[
 u_n(x)- \varepsilon\le u(x)\le u_n(x)+ \varepsilon\quad\forall x\in\T, \quad\forall n\ge n_0.
\]
Therefore the sequence $\{u_n\}_{n\in\N}$ converges uniformly to $u.$
\end{proof}

\begin{ex} Case $p=\infty.$ Let $m=3,$ $\alpha=1,$ $\beta=0$ and
$F:[0,1]\to\R$ given by $F(t)=t.$ In \cite{s-tree}, the author
proves that the solution of $(DP)$ with boundary data $F$ is
    \[
    u(x)=\int_{I_{x}}t\, d\mathcal{C}^x(t)\qquad \forall x\in\mathbb{T}_3,
    \]
    where $\mathcal{C}^x$ is the Cantor measure on the interval $I_x$ with $\mathcal{C}^x(I_x)=1.$

%
%
    \end{ex}

\begin{ex} Case $p=2.$ Let $m=3,$ $\alpha=0,$ $\beta=1$ and  $F:[0,1]\to\R$ given by
$F(t)=(t-\nicefrac{1}{2})^2.$ In this case, the solution $u$ of
$(DP)$ is
    \[
    u(x)=\frac{1}{|I_x|}\int_{I_{x}}\left(t-\frac12\right)^2\, dt\quad \forall x\in\mathbb{T}_3,
    \]
where $|I_x|$ is the  measure of  $I_x.$

%
%
    \end{ex}

\section{Unique continuation property} \label{sect-unique.cont}

In this section we prove our main result that deals with subsets of $\T$ that have the unique
continuation property.

\begin{de}
We say that a subset $U$ of $\T$ satisfies the unique continuation
property $(UCP)$ if for any bounded $p-$harmonious function $u$
such that $u=0$ in $U,$ we have that $u\equiv 0$ in $\T.$
\end{de}

Let us first prove that the density of the set $\psi(U)$ in
$[0,1]$ is a necessary condition for $UCP$.

\begin{te}
If $U\subset \T$  satisfies UCP then $\psi(U)$ is dense in
$[0,1].$
\end{te}

\begin{proof}
We will show that if $\psi(U)$ is not dense in $[0,1],$  then
there exists a $p-$harmonious function $u$ such that $u\neq0$ in
$\T$ and $u=0$ in $U$.

Since $\psi(U)$ is not dense in $[0,1]$
there exist $\tau>0$ and $r\in[0,1]$ such that
\begin{equation}\label{vacio}
(r-\tau,r+\tau)\cap\psi(U)=\emptyset.
\end{equation}
Then there exist $k\in\N$ and $x=(a_1,\dots,a_k)\in\T$ such that
$\nicefrac{1}{m^k}<\tau$ and $I_x\subset(r-\tau,r+\tau).$
Therefore, using
\eqref{vacio} and the fact that $I_x$ is the subset of $\partial\T$
consisting of all branches that start at $x,$ we have that
$(x,b_1,\dots,b_s)\notin U$ for all $s\in\N.$ Now, we construct
$u$ as follows
\[
u(y)=
   \begin{cases}
  \, \,  \,   \,  1 & \forall y\in\T\text{  such that  } I_y\subset I_{(x,0)},  \\
     -1 &\forall y\in\T\text{  such that  } I_y\subset I_{(x,m-1)}, \\
   \, \,  \, \,  0 & \text{ otherwise}.
  \end{cases}
\]
It is clear that $u$ is a bounded $p-$harmonious function such
that $u=0$ in $U$ and   $u\neq0.$ This finishes the proof.
\end{proof}

\begin{pr}\label{PA}
Let $U$ be a subset of $\T$. If $U$ satisfies the following property
\begin{enumerate}
    \item[{\rm(PA)}] There exists $n\in\N$ such that for all $x\in\T$ there exist $l\in\{1,\dots,n\}$ and
        at least one branch starting at $x$ such that its $l-$th node belongs to $U,$
\end{enumerate}
then $U$ satisfies UCP.
\end{pr}
\begin{re} Let $U$ be a subset of $\T.$
It is easy to see that if  $U$ satisfies PA, then $\psi(U)$ is dense in $[0,1].$
\end{re}
\begin{proof}[Proof of Proposition \ref{PA}]
Let $u$ be a bounded $p-$harmonious function such that $u=0$ in
$U$. Set $M= \sup\{u(x)\colon x\in\T\}$ and $\delta=
\left(\nicefrac{\alpha}{2} + \nicefrac{\beta}{m}\right)$. Given
$\varepsilon>0$ there exists $x_0\in\T$ such that $u(x_0)\geq
M-\varepsilon.$ Thus, since $u$ is a $p-$harmonious function, we
have that

\begin{align*}
M-\varepsilon& \leq u(x_0)=
\dfrac{\alpha}{2}\left\{\max_{y\in \S(x_0)} u(y)
+ \min_{y\in \S(x_0)} u(y)\right\}+ \dfrac{\beta}{m}
\sum_{y\in \S(x_0)} u(y)\\
&\leq \left(\dfrac{\alpha}{2}+ \frac{m-1}{m}\beta\right)M
+  \left(\frac{\alpha}{2}+
\dfrac{\beta}{m}\right)\min_{y\in \S(x_0)} u(y).
\end{align*}

Then
\[
M-\frac{\varepsilon}{\delta}\leq \min_{y\in \S(x_0)} u(y)\leq u(y)
\]
for all $y\in \S(x_0).$

On the other hand, since $U$ satisfies PA, there exist
$l\in\{1,\dots,n\}$ and $(x_0, a_1, \dots, a_l)\in U$ where $a_k\in \{0, \dots, m-1\}$ for all $1\le
k\le l.$
Then, using that $x_1=(x_0,a_1)\in\S(x_0)$ and the above inequality, we get
\[
 M-\frac{\varepsilon}{\delta}\le u(x_1).
 \]
Similarly, we have
\[
M-\frac{\varepsilon}{\delta^k}\le u(x_k) \quad\forall k\in\{2,\dots,l\}
\]
where $x_k=(x_{k-1},a_k),$  $2\le k\le l.$ Then, using that
$x_l=(x_{l-1}, a_l)=(x_0, a_1, \dots, a_l)\in U,$ we have that
\[
M\delta^l\le \varepsilon.
\]
Let us now suppose that $M\ge 0.$ Using that $l\le n,$
$0<\delta<1$ and the above inequality, we have that
\[
M\delta^n\le \varepsilon \quad \forall \varepsilon>0,
\]
then $M=0.$ Thus we have that $M\le 0.$

In the same manner we can show that $N=\inf\{u(x)\colon
x\in\T\}\ge 0$. Therefore, $M=N=0,$ which proves the theorem.
\end{proof}

\begin{de}
Let $U$ be a subset of $\T$ such that $\T^n\setminus U\neq
\emptyset$ for all $n\in \N$. We define the sequence
$\{\rho_k(U)\}_{k\in\N}\subset\N$ as follows:
\[
\rho_1(U)\colon=\min\{n\in\N\colon \exists x\in\T^n\cap U\},
\]
and for all  $k\in\N_{\ge 2}$,
\[
\rho_k(U)\colon=\min\{n\in\N\colon \exists y\in\T^{\eta_{k-1}(U)}
\setminus U \mbox{ and } x\in\T^{\eta_{k-1}(U)+n}\cap U
\mbox{ s. t. } I_x\subset I_y\},
\]
where $$\eta_{k-1}(U)=\dis\sum_{j=1}^{k-1}\rho_{j}(U).$$ In
addition, for all $k\in\N_{\ge2},$ we define the sets
\[
\mathcal{A}_k(U)\colon=\left\{y\in\T^{\eta_{k-1}(U)}\setminus U
\colon I_{y}\cap I_{x_j}=\emptyset, \
x_j\in
\T^{\eta_{j}(U)}\cap U, \ \forall j\in\{1,\dots,k-1\}\right\}.
\]
We will write simply $\rho_k,$ $\eta_{k-1}$ and $\mathcal{A}_k$
when no confusion arises.
\end{de}

We can now formulate our main result.

\begin{te}\label{P1P2}
Let $U$ be a subset of $\T$ such that $\psi(U)$ is dense in
$[0,1]$, $\T^n\setminus U\neq \emptyset$ for all $n\in \N$ and
$U$ satisfies the following properties
\begin{enumerate}
    \item[{\rm(P1)}] There exists a unique $x_1\in U\cap\T^{\rho_1}.$
    \item[{\rm(P2)}] For all $k\in\N_{\ge2}$ and  for all $y\in\mathcal{A}_{k}$ there exists a
    unique  \mbox{$x\in\T^{\eta_{k-1} + \rho_{k}}\cap U$}  such that $I_{x}\subset I_{y}.$
\end{enumerate}
Then $U$ satisfies UCP  if only if
$$
\sum_{k=1}^{\infty} \delta^{\rho_k} = +\infty
$$
where $\delta=1-\theta,$  $\theta= \frac{\alpha}{2}+\frac{m-1}{m}\beta.$
\end{te}

\begin{proof}
We will proceed in two steps.

\noindent {\bf Step 1.} First we will prove that if $U$ satisfies UCP,  then
$$
\sum_{k=1}^{\infty} \delta^{\rho_k} = +\infty.
$$
Arguing by contradiction, we suppose  that
$
\sum_{k=1}^{\infty} \delta^{\rho_k} < +\infty$.
By (P1), there exists a unique $x_1=(a_1,\dots,a_{\rho_1})\in U$
such that $\tau_{1i}=(a_1,\dots,a_i)\notin U$ for any
$1\le i<
\rho_1.$ We now construct a $p-$harmonious function $u$ such
that
$u=0$ in $U$ as follows:
\begin{align*}
&u(\emptyset)=1, \\
&u(a_1)=m_{11} =\displaystyle \min_{y\in S(\emptyset)}u(y)\\
&u(b_1, \dots, b_j)=M_{11} =
\displaystyle \max_{y\in \S(\emptyset)}u(y)
\mbox{ if } b_1\neq a_1\quad \forall 1\le j\le \rho_1,
\end{align*}
and for any $2\le i< \rho_1$
\begin{align*}
&u(\tau_{1i})=m_{1i} =\displaystyle \min_{y\in S(\tau_{1(i-1)})}u(y)  \\
&u(\tau_{1(i-1)},b_i,\dots,b_j)=M_{1i} =\displaystyle \max_{y\in \S(\tau_{1(i-1)})}u(y)
\mbox{ if } b_i\neq a_i\quad  \forall i\le j\le \rho_1.
\end{align*}
Since $x_1\in U$ and we need that $u=0$ in $U,$ we define
\[
u(x_1)=0=m_{1\rho_1}=\min_{y\in S(\tau_{1(\rho_1-1)})}u(y).
\]
We also take $u(y)=0$ for all $y\in\T$ such that $I_y\subset
I_{x_1}.$ Thus, in order for $u$ to be  a $p-$harmonious function,
we need to take $M_{11}, \dots, M_{1\rho_1}$  and $m_{11}, \dots,
m_{1(\rho_1-1)}$ such that
\begin{align*}
 1&= \frac{\alpha}{2}(M_{11} + m_{11})+ \frac{\beta}{m} \left((m-1)M_{11} +m_{11}\right),\\
m_{1i}&= \frac{\alpha}{2}(M_{1(i+1)} + m_{1(i+1)}) +\frac{\beta}{m}
\left((m-1)M_{1(i+1)} +m_{1(i+1)}\right)\quad \forall 1\le i <\rho_1.
\end{align*}
Then, we can observe that
\begin{align}
 1&= \frac{\alpha}{2}(M_{11} + m_{11})+\frac{\beta}{m} \left((m-1)M_{11} +m_{11}\right) \notag\\
 \label{uno} &= \left(\frac{\alpha}{2}+\frac{m-1}{m}\beta\right)M_{11}
 + \left(\frac{\alpha}{2}+\frac{\beta}{m} \right)m_{11}\\
 &= \theta M_{11}+ (1- \theta)m_{11}\notag
\end{align}
and in the same manner, we can show that
\begin{equation}\label{ig1}
m_{1i}= \theta M_{1(i+1)}+ (1-\theta)m_{1(i+1)} \quad \forall 1\leq i<\rho_1 .
\end{equation}
Now, using that $m_{1\rho_1}=0,$ we have that
\[
M_{1\rho_1}=\frac{m_{1(\rho_1-1)}}{\theta}.
\]
\begin{center}

\usetikzlibrary{decorations.pathreplacing}
\begin{tikzpicture} [font=\footnotesize,grow=down, 
level 1/.style={->, sibling distance=3em},
level 2/.style={->,sibling distance=3em}, level distance=1cm,
level 3/.style={->,sibling distance=3em}, level distance=1cm]
  \node(b){1}
        child { node {$M_{11}$}}
        child { node {$M_{11}$}}
       child { node {$m_{11}$}
					child { node {$M_{12}$}}
        			child { node {$M_{12}$}}
                     child { node(a){$m_{12}$}
																		}
				}    ;
\draw[dotted,black] (a) -- (3.5,-2.9);
  \node at (3.7,-3.1){$m_{1(\rho_1-1)}$}
        child { node {$\frac{m_{1(\rho_1-1)}}{\theta}$}}
        child { node {$\frac{m_{1(\rho_1-1)}}{\theta}$}}
       child { node {0}}    ;
\draw [decorate,decoration={brace,amplitude=5pt},xshift=-4pt,yshift=0pt] (5.5,0.18) -- (5.5,-4.3) node [black,midway,xshift=0.5cm] {$\rho_1$};

\end{tikzpicture}
\end{center}
If we take
\begin{equation}\label{ig2}
    M_{1i}= M_{1\rho_1}=\frac{m_{1(\rho_1-1)}}{\theta}=M_1, \quad\forall 1\le i\le \rho_1,
\end{equation}
by \eqref{ig1},  we obtain
\[
m_{1i}= m_{1(\rho_1-1)} + (1-\theta) m_{1(i+1)} \quad \forall  1\leq i<\rho_1.
\]
Using the above equality, we have
\[
 m_{1(\rho_1-2)}= m_{1(\rho_1-1)} + (1-\theta)m_{1(\rho_1-1)}= (2-\theta)m_{1(\rho_1-1)},
 \]
 and, for any $2< j\leq \rho_1-1,$
 \begin{equation}
\label{m11}m_{1(\rho_1-j)}= \left(\sum_{k=0}^{j-3} (1-\theta)^k
+ (1-\theta)^{j-2}(2-\theta)\right)m_{1(\rho_1-1)}.
\end{equation}
Thus, by \eqref{uno} and \eqref{m11}, we have that
\[
m_{1(\rho_1-1)}= \frac{1}{\displaystyle\sum_{k=0}^{\rho_1-3}
(1-\theta)^k+ (1-\theta)^{\rho_1-2}(2-\theta)}.
\]
In addition, since $M_1= \frac{m_{1(\rho_1-1)}}{\theta},$ we obtain
\[
M_1=  \frac{1}{\theta\left(\displaystyle\sum_{k=0}^{\rho_1-3}
(1-\theta)^k+ (1-\theta)^{\rho_1-2}(2-\theta)\right)}.
\]
Then, taking $\delta= 1-\theta,$ we get
\[
\theta\left(\displaystyle\sum_{k=0}^{\rho_1-3} (1-\theta)^k+ (1-\theta)^{\rho_1-2}(2-\theta)\right)
 = (1-\delta)\displaystyle\sum_{k=0}^{\rho_1-1} \delta^{k}
 = 1-\delta^{\rho_1}.
\]
Therefore,
$$
M_{1}= \frac{1}{ 1-\delta^{\rho_1}}.
$$

On the other hand,
\[
\mathcal{A}_2=\{y_j\}_{j=1}^{m^{\rho_1}-1}\mbox{ and }
u(y_j)=M_{11}\quad\forall  j\in\{1,\dots,m^{\rho_1-1}\}.
 \]
Furthermore, by (P2), for all $j\in\{1,\dots, m^{\rho_1}-1\}$
there exists a unique
\[
x_2^j=(y_j,a_{\rho_1+1}^j,\dots,a_{\rho_1+\rho_2}^j)\in\T^{\rho_1 + \rho_{2}}\cap U
\]
 with $\tau_{2i}^j=(y_j,a_{\rho_1+1}^j,\dots,a_{\rho_1+i}^j)\notin U$ for any $i\in\{1,\dots, \rho_2\}.$

Let $j\in\{1,\dots, m^{\rho_1}-1\}.$ We define $u$ as follows
\begin{align*}
&u(y_j,a_{\rho_1+1}^j)=m_{21} =\displaystyle \min_{y\in S(y_j)}u(y)\\
&u(y_1,b_{\rho_1+1}, \dots, b_{\rho_1+l})=M_{21} =\displaystyle
\max_{y\in \S(y_j)}u(y)\mbox{ if } b_{\rho_1+1}\neq a_{\rho_1+1}^j\quad \forall  l\in\{1, \dots \rho_2\},
\end{align*}
and for any $2\le i< \rho_2,$
\begin{align*}
&u(\tau_{2i}^j)=m_{2i} =\displaystyle \min_{y\in \S(\tau_{2(i-1)}^j)}u(y), \\
&u(\tau_{2(i-1)}^j,b_{\rho_1+i},\dots,b_{\rho_1+j})=M_{2i} =
\displaystyle
\max_{y\in \S(\tau_{2(i-1)}^j)}u(y)
\mbox{ if }  b_{\rho_1+l}\neq a_{\rho_1+i}^j\quad \forall l\in\{ i,\dots, \rho_2\}.
\end{align*}
Since $x_2^j\in U$ and we need that $u=0$ in $U,$ we define
\[
u(x_2^j)=0=m_{2\rho_2}=\min_{y\in \S(\tau_{2(\rho_2-1)}^j)}u(y).
\]
We also take $u(y)=0$ for all $y\in\T$ such that $I_y\subset I_{x_2^j}.$

Arguing as before, taking
\begin{equation*}
    M_{2i}= M_{2\rho_2}=\frac{m_{2(\rho_2-1)}}{\theta}=M_2, \quad\forall 1\le i\le \rho_2,
\end{equation*}
we get
\begin{align*}
m_{2(\rho_2-1)}=&\frac{M_{1}}{\displaystyle\sum_{k=0}^{\rho_2-3}
(1-\theta)^k+ (1-\theta)^{\rho_2-2}(2-\theta)},\\
m_{2(\rho_2-l)}=& \left(\sum_{k=0}^{l-3} (1-\theta)^k+ (1-\theta)^{j-2}
(2-\theta)\right)m_{2(\rho_2-1)} \ \forall
l\in\{2,\dots,\rho_2-1\},
\end{align*}
and
\[
M_{2}= \frac{M_{1}}{1-\delta^{\rho_2}}= \frac{1}{(1-\delta^{\rho_1})(1-\delta^{\rho_2})}.
\]

\medskip

By induction in $k,$ we construct  $u$ so that $u$ is
$p-$harmonious in $\T$ such that $u=0$ in $U,$ $u\neq0$ in $\T$
and
  \[
 M_k=\prod_{i=1}^{k} \frac{1}{1-\delta^{\rho_i}} \quad \forall k\in\N.
 \]

Since
\[
\sum_{k=1}^{\infty} \delta^{\rho_k} <+\infty,
\]
we have that
\[
\sum_{i=1}^{\infty} \frac{\delta^{\rho_k}}{1-\delta^{\rho_k}}<+\infty\Leftrightarrow
\sum_{k=1}^{\infty} \log\left(1+\frac{\delta^{\rho_k}}{1-\delta^{\rho_k}}\right)
=\sum_{k=1}^{\infty} \log\left(\frac{1}{1-\delta^{\rho_k}}\right)<+\infty.
\]
Thus,
 \[
 \prod_{i=1}^{\infty} \frac{1}{1-\delta^{\rho_k}}<+\infty.
 \]
 Therefore $u$ is a bounded $p-$harmonious function such that $u=0$ in $U$
  and $u\neq0$ in $\T$.
This is a contradiction.

\noindent {\bf Step 2.} We assume that
\[
\sum_{i=1}^{\infty} \delta^{\rho_i} = +\infty
\]
and we will prove that $U$ satisfies the $UCP$.

Suppose that there exists a $p-$harmonious function $v\neq0$ such
that $v=0$ in $U.$ We will prove that $v$ is unbounded.
Multiplying $v$ by a suitable constant, we can assume that
$v(\emptyset)=1.$ Let $u$ be defined as in the above step. First,
we need to show that
\begin{equation}\label{induc}
M_k\le\max\{v(y)\colon y\in\T^{\rho_k} \}\qquad\forall k\in\N.
\end{equation}
To this end, we observe that
\[
\theta M_1+(1-\theta)m_{11}=u(\emptyset)=1=v(\emptyset)\le
\theta \max_{y\in\S(\emptyset)}v(y)+(1-\theta)\min_{y\in\S(\emptyset)}v(y),
\]
then
\[
M_1\le\max_{y\in\S(\emptyset)}v(y)\quad \mbox{ or }\quad m_{11}\le \min_{y\in\S(\emptyset)}v(y).
\]

If $M_1\le\dis \max_{y\in\S(\emptyset)}v(y)$ then
$M_1\le \max\{v(y)\colon y\in\T^k \mbox{ with } k\in\{1,\dots,\rho_1\}\},$
and therefore $M_1\le\max\{v(y)\colon y\in\T^{\rho_1}\}.$

Now we consider the case $M_1>\dis \max_{y\in\S(\emptyset)}v(y)$
and $\dis m_{11}\le \min_{y\in\S(\emptyset)}v(y).$

By (P1), there
exists a unique $x_1=(a_1,\dots,a_{\rho_1})\in U$ such that
$\tau_{1i}=(a_1,\dots,a_i)\notin U$ for any $1\le i<
\rho_1.$ Then, since $m_{11}\le\dis
\min_{y\in\S(\emptyset)}v(y)\le v(a_1),$ we have that
\[
\theta M_1+(1-\theta)m_{12}=m_{11}\le v(a_{1})\le
\theta \max_{y\in\S(a_1)}v(y)+(1-\theta)\min_{y\in\S(a_1)}v(y),
\]
and then
\[
M_1\le\max_{y\in\S(a_1)}v(y)\quad \mbox{ or }\quad m_{12}\le \min_{y\in\S(a_1)}v(y).
\]

Again, if $M_1\le\dis \max_{y\in\S(a_1)}v(y),$ then we have that
$M_1\le\max\{v(y)\colon y\in\T^{\rho_1}\}.$ If $m_{12}\le\dis
\min_{y\in\S(a_1)}v(y)\le v(\tau_{12}),$ then we can prove as
before that
\[
M_1\le\max_{y\in\S(\tau_{12})}v(y)\quad \mbox{ or }\quad m_{13}\le \min_{y\in\S(\tau_{12})}v(y).
\]

In the same manner, using $\rho_1-1$ steps, we show that
\[
 M_1\le\max\{v(y)\colon y\in\T^{\rho_1}\} \mbox{ or }  m_{1(\rho-1)}\le \min_{y\in\S(\tau_{1(\rho-2)})}v(y).
\]
If $  m_{1(\rho-1)}\le\dis \min_{y\in\S(\tau_{1(\rho-2)})}v(y)\le v( \tau_{1(\rho-1)}),$  then
\[
\theta M_1=m_{1\rho_1}\le v(\tau_{1(\rho-1)})\le\theta
\max_{y\in\S(\tau_{1(\rho-1)})}v(y) +(1-\theta)\min_{y\in\S(\tau_{1(\rho-1)})}v(y).
\]
Since $x_1=(\tau_{1(\rho-1)},a_{\rho_1})\in U$ and $v=0$ in $U,$
$\dis\min_{y\in\S(\tau_{1(\rho-1)})}v(y)\le0$ and then
\[
M_1\le \max_{y\in\S(\tau_{1(\rho-1)})}v(y).
\]
Therefore
\[
 M_1\le\max\{v(y)\colon y\in\T^{\rho_1}\}.
\]
Then, by induction on $k,$ using (P2), we have that \eqref{induc}
holds.

Since
\[
\sum_{j=1}^{\infty} \delta^{\rho_j} = +\infty,
\]
we have that
\[
\lim_{k\to+\infty}M_k = \lim_{k\to+\infty}\prod_{i=1}^{k}\frac{1}{1-\delta^{\rho_i}}=+\infty.
\]
Therefore, by \eqref{induc}, $v$ is an unbounded.
The proof is complete.
\end{proof}

\subsection{Examples} Below we give some examples of sets
verifying (or not) the $UCP$.

\begin{ex} Let $U$ be given by
\[
U=\bigcup_{k\in\N} \T^{2^k}.
\]
Then it is clear that $U$ has the $UCP$.
\end{ex}

\begin{ex}
Let $m=3$ and $U$ be given by
\[
U= \left\{x \in\mathbb{T}_3\, \colon \,  x=(a_1,a_2,\dots,a_n), \
a_i\neq 1 ,\, \forall 1\le i\le n
\right\}.
\]
It is easy to see that $\psi(U)$ is a Cantor set and therefore
$U$ does not have the $UCP$.
\end{ex}

\begin{ex} Let $U$ be given by
$$
U= \left\{x \in\T\colon   x=(a_1,a_2,\dots,a_n), \ a_n=0
\right\}.
$$
Then, since $U$ satisfies (PA) with $n=1$, $U$ has the $UCP$.
\end{ex}

\begin{ex}
Let $U_1:=\{(0)\},$ $\rho_1:=1$
\[
U_{2n}:=\{x\in\T^{\mu_{2n-1}+2^{n+1}}\colon x=(y,a_1,\dots,a_{2^{n+1}}) \colon
y\in\T^{\mu_{2n-1}}\setminus U_{2n-1}\}, \rho_{2n}:=2^{n+1}
\]
\[
U_{2n+1}:=\{x\in\T^{\mu_{2n}+1}\colon x=(y,0) \colon
y\in\T^{\mu_{2n}}\setminus U_{2n}\}\mbox{ and } \rho_{2n+1}:=1
\]
for all $n\in\N,$ where $\mu_n:=\sum_{j=1}^n\rho_j$ for all
$n\in\N.$

Then $U=\bigcup_{n\in\N}U_n$ is dense and satisfies (P1)
and (P2). Since $
\sum_{j=1}^{+\infty}\rho_j=\infty,$ by Theorem \ref{P1P2}, $U$
satisfies the $UCP$.
\end{ex}

\end{document}